\documentclass[10pt]{amsart}
\usepackage{latexsym, amsmath, amsfonts, amssymb, mathrsfs, fancyhdr, tikz, tikz-cd}
\usepackage{verbatim}
\usepackage{multicol}
\usetikzlibrary{matrix,arrows,decorations.pathmorphing,calc}

\makeatletter
\@namedef{subjclassname@2020}{%
  \textup{2020} Mathematics Subject Classification}
\makeatother

\def\ds{\displaystyle}

\def\R{\mathbb{R}}

\def\e{\varepsilon}
\def\s{\sigma}

\def\d{\delta}

\def\l{\lambda}

\def\S{\mathcal{S}}

\def\a{\alpha}

\def\g{\gamma}

\def\y{\mathfrak{h}}
\def\c{\mathfrak{c}}

\def\ch2{\mathbb{C} \mathbb{H}^2}
\def\h2{\mathbb{H}^2}
\def\ddt{\frac{\partial}{\partial \theta}}
\def\ddr{\frac{\partial}{\partial r}}
\def\th{\theta}

\def\oshr/2{\cosh \left( \frac{r}{2} \right)}
\def\inhr/2{\sinh \left( \frac{r}{2} \right)}
\def\osh2r/2{\cosh^2 \left( \frac{r}{2} \right)}
\def\inh2r/2{\sinh^2 \left( \frac{r}{2} \right)}
\def\-1/4{- \frac{1}{4}}
\def\H{\mathbb{H}}
\def\C{\mathbb{C}}
\def\S{\mathbb{S}}

\newtheorem{theorem}{Theorem}[section]
\newtheorem{lemma}[theorem]{Lemma}
\newtheorem{cor}[theorem]{Corollary}

\theoremstyle{definition}

\theoremstyle{remark}
\newtheorem{remark}[theorem]{Remark}

\numberwithin{equation}{section}

\begin{document}

\title{K\"{a}hler manifolds with an almost $1/4$-pinched metric}

\author{Barry Minemyer}
\address{Department of Mathematical and Digital Sciences, Commonwealth University - Bloomsburg, Bloomsburg, Pennsylvania 17815}
\email{bminemyer@commonwealthu.edu}



\date{\today.}



\begin{abstract}
In this paper we construct an almost negatively $1/4$-pinched Riemannian metric on a class of compact manifolds recently discovered by Stover and Toledo in \cite{ST}.  
It is known that these manifolds are K\"{a}hler and not locally symmetric.  
These are the first known examples of not locally symmetric K\"{a}hler manifolds admitting such a metric and, via the result of Hernandez \cite{Hernandez} and Yau and Zheng \cite{YZ}, these manifolds cannot admit a negatively quarter-pinched Riemannian metric.
This metric is also interesting because it is a generalization to the complex hyperbolic setting of the famous pinched metric constructed by Gromov and Thurston in \cite{GT}.  
\end{abstract}

\maketitle



\section{Introduction}\label{Section:Introduction}

Via Mostow rigidity \cite{Mostow} one knows that any compact Riemannian manifold of dimension $n \geq 3$ with constant sectional curvature equal to $-1$ must be isometric to a quotient of hyperbolic space $\H^n$.  
In \cite{GT} Gromov and Thurston show that this assumption on the curvature cannot be relaxed at all by constructing a family of branched cover manifolds that are not homotopy equivalent to a hyperbolic manifold, but which admit a Riemannian metric with all sectional curvatures lying in the interval $(-1-\e, -1+\e)$ for any prescribed $\e > 0$.  
Via Ontaneda's smooth hyperbolization \cite{Ontaneda}, one now obtains many more examples of pinched negatively curved Riemannian manifolds.  
These results demonstrate that there is generally much ``flexibility" within the setting of pinched negatively curved Riemannian manifolds.  

In contrast, our current knowledge of negatively curved compact K\"{a}hler manifolds indicates that they are incredibly rigid.  
Mostow rigidity also applies to the complex hyperbolic setting where, in this paper, we scale the metric in complex hyperbolic space $\C \H^n$ to have constant holomorphic sectional curvature $-4$.  
Thus, all sectional curvatures lie in the interval $[-4,-1]$, which is commonly referred to as being negatively 1/4-pinched.    
Mostow rigidity shows that if a compact K\"{a}hler manifold of complex dimension $n \geq 2$ has constant holomorphic sectional curvature of $-4$ (with respect to the K\"{a}hler metric), then this manifold is isometric to a quotient of $\C \H^n$.  

It is natural to ask if the curvature assumption here can be weakened and, as far as the author is aware, this question is still open.  
There are no known examples of compact manifolds which admit a K\"{a}hler metric with all holomorphic sectional curvatures lying within $(-4-\e, -4+\e)$ for sufficiently small $\e > 0$ besides the locally symmetric ones.  
But from a result of Deraux and Seshadri \cite{DS} one does know that, if such a compact K\"{a}hler manifold $M$ does exist, then it must be ``close" to being complex hyperbolic in the sense that the ratio of all Chern numbers of $M$ must be close to the corresponding ratio of Chern numbers of a complex hyperbolic manifold.  

Shockingly, the curvature assumption for the complex hyperbolic form of Mostow rigidity cannot be weakened for general Riemannian metrics.  
A result proved independently by Hernandez \cite{Hernandez} and Yau and Zheng \cite{YZ} says that, if a compact K\"{a}hler manifold $M$ admits a (a priori not necessarily K\"{a}hler) metric $g$ that is negatively 1/4-pinched, then $(M,g)$ is isometric to a quotient of $\C \H^n$.  
In particular, $g$ must be K\"{a}hler.  
If a Riemannian metric on a compact K\"{a}hler manifold had {\it all} sectional curvatures (not just holomorphic sectional curvatures) pinched near $-4$, then by scaling the metric one would have all sectional curvatures within $[-4,-1]$ and thus this manifold would be isometric to a (scaled) quotient of $\C \H^n$.  

The aforementioned results of Mostow, Hernandez, Yau and Zheng, and Dereaux and Seshadri paint a picture of incredible rigidity within compact K\"{a}hler manifolds with negative curvature.  
Indeed, prior to a recent breakthrough by Stover and Toledo \cite{ST}, there were very few known examples of such manifolds.  
The only such constructions that the author is aware of are by Mostow and Siu \cite{MS}, Hirzebruch \cite{Hirzebruch} (whose K\"{a}hler metric was constructed by Zheng \cite{Zheng}), Zheng \cite{Zheng2}, and Deraux \cite{Deraux}.  
The largest complex dimension of any of these examples is $n=3$ and, by the result of Deraux and Seshadri \cite{DS} mentioned above, it is known that none of the metrics on these examples are close to 1/4-pinched.  

In \cite{ST} Stover and Toledo prove the existence of a large family of complex hyperbolic branched cover manifolds in every dimension analogous to the manifolds constructed in \cite{GT}.  
Stover and Toledo prove that these manifolds are not homotopy equivalent to a quotient of $\C \H^n$ and, by the work of Zheng \cite{Zheng}, these manifolds admit a K\"{a}hler metric with negative sectional curvature.  
As mentioned above, the metric constructed by Zheng is not close to 1/4-pinched.  
Our main result is to construct an almost 1/4-pinched Riemannian metric on the Stover-Toledo manifolds, thus showing that the curvature assumptions in \cite{Hernandez} and \cite{YZ} cannot be relaxed and providing the first example of flexibility within the class of negatively curved K\"{a}hler manifolds.


\begin{theorem}\label{thm:main theorem}
For any $\e > 0$ and every integer $n \geq 2$, there exist compact K\"{a}hler manifolds of complex dimension $n$ that admit a Riemannian metric that is $\e$-close to being negatively $1/4$-pinched but are not homotopy equivalent to a locally symmetric manifold.  
\end{theorem}

A few remarks about the Theorem.  
As mentioned in the abstract, the Riemannian metric in Theorem \ref{thm:main theorem} is, in some sense, a generalization of the famous Gromov-Thurston metric in \cite{GT} to the complex hyperbolic setting.  
But the construction of the metric in Theorem \ref{thm:main theorem} is very different.  
Starting with the standard complex hyperbolic metric and simply ``increasing the angle" as in the Gromov-Thurston case will generally not result in pinched (or even negative) sectional curvature.  
See Remark \ref{rmk:GT difference} below for a more detailed discussion about why this is the case.  
It should be noted that one would not expect the Riemannian metric from the Theorem to be K\"{a}hler.
The reason for this will be discussed in the next Subsection.  
Finally, one should also note that Farrell and Jones construct exotic manifolds with almost negatively $1/4$-pinched metrics in \cite{FJ}.  
The manifolds in these examples are homeomorphic but not diffeomorphic to complex hyperbolic manifolds.  
It is not obvious to the author that the manifolds constructed by Farrell and Jones are still K\"{a}hler since they have a different smooth structure, but in any case Theorem \ref{thm:main theorem} still provides the first examples of almost negatively $1/4$-pinched metrics on K\"{a}hler manifolds that are not homotopy equivalent to complex hyperbolic manifolds.  

The new construction in this paper is the almost negatively $1/4$-pinched metric in Theorem \ref{thm:main theorem}.
The remainder of the Theorem can be obtained from the literature as follows.  
Following Stover and Toledo in \cite{ST}, let $\Gamma < \text{Isom}(\C \H^n)$ be a cocompact congruence arithmetic lattice of simple type, and let $M = \Gamma \setminus \C \H^n$.  
It is known that such a manifold $M$ contains an immersed (complex) codimension one totally geodesic submanifold $N$.  
One can then find a finite congruence cover $(M', N')$ of $(M,N)$ with $N'$ embedded in $M'$. 
Let $\Gamma'$ be such that $M' = \Gamma' \setminus \C \H^n$.   
Our almost $1/4$-pinched Riemannian metric will require a large normal injectivity radius about the totally geodesic submanifold.  
Suppose that, for a fixed $\e > 0$, we require a normal injectivity radius of $R$.  
By \cite{DKV} $M'$ contains only finitely many homotopy classes of closed geodesics with length less than $R$.  
Since $\Gamma'$ is residually finite, we can pass to a finite index subgroup $\Gamma''$ of $\Gamma'$ where we remove elements of $\Gamma'$ that correspond to each of these closed geodesics.  
In the resulting pair $(M'', N'')$, $N''$ will still be embedded and will now have normal injectivity radius of at least $R$.  
Let us quickly note that this is the only place where we require $M$ to be compact.  
If the submanifold $N$ can be constructed to have a sufficiently large normal injectivity radius, then the ambient manifold $M$ in this construction need not be compact.  

By Theorem 5.1 of \cite{ST} there exists an integer $d > 2$ and a further finite cover $(\tilde{M}, \tilde{N})$ of $(M'', N'')$ such that the $d$-fold ramified branched cover $X$ of $\tilde{M}$ about $\tilde{N}$ is a smooth manifold.  
This manifold $X$ is the manifold described in Theorem \ref{thm:main theorem}.  
Note that $\tilde{N}$ will still have a normal injectivity radius of at least $R$ within $\tilde{M}$, and therefore the ramification locus will still have a normal injectivity radius of at least $R$ in $X$.  
As noted above, Stover and Toledo in \cite{ST} prove that $X$ is not homotopy equivalent to a locally symmetric manifold (Theorem 1.5 in \cite{ST}) and Zheng proves that $X$ admits a (negatively curved) K\"{a}hler metric.  


Lastly, Stover and Toledo do not compute the Chern numbers of $X$ when proving that $X$ is not homotopy equivalent to a locally symmetric manifold.  
They instead give an argument based on the normal bundle of the ramification locus.  
So it is unknown if the bounds on the difference between the ratio of Chern numbers of $X$ and of a complex hyperbolic manifold developed in \cite{DS} are satisfied.  
It seems likely to the author that there exists an example of a branched cover $X$ which satisfies Theorem \ref{thm:main theorem} but whose ratio of chern numbers is not close to that of a complex hyperbolic manifold, but again this is not immediately clear.

\subsection*{Outline of the construction of the metric $g$}
In what follows, all constructions are at the level of the universal cover.  
The constructed metric will descend to the manifold $X$ as long as we end up with the pullback metric by the end of the normal injectivity radius of the ramification locus.  
We also assume that the ramification locus is connected, and thus is a copy of $\C \H^{n-1}$ in the universal cover.  
If it is disconnected, then we just apply the constructed metric about each component.  

Let us quickly review the metric constructed by Gromov and Thurston in \cite{GT}.  
Let
	\begin{equation*}
	\l =  h^2 \y_{n-2} + v^2 d \th^2 + dr^2.
	\end{equation*}
where $\y_{n-2}$ denotes the hyperbolic metric on the core copy of $\H^{n-2}$ and $h$ and $v$ are positive functions of $r$.
The $h^2 \y_{n-2}$ component of $\l$ is the horizontal fiber and is tangent to the $r$-tube about the base copy of $\H^{n-2}$.  
It is well-known that $\l$ is the hyperbolic metric $\y_n$ on $\H^n$ written in polar coordinates about a totally geodesic codimension two copy of $\H^{n-2}$ if we choose the warping functions $h(r) = \cosh(r)$ and $v(r) = \sinh(r)$.  
If we let $h(r) = \cosh(r)$ and $v(r) = d \sinh(r)$ then $\l$ will be the pullback metric on the entire branched cover, which will be everywhere Riemannian except on the ramification locus.  
At these points the metric will have a cone angle of $2 d \pi$.  

If one chooses an orthonormal basis for $\l$ at any point and computes the components of the curvature tensor (Theorem \ref{thm:curvature formulas for h} below), one first sees that all mixed terms vanish.  
In the curvature formulas for the coordinate planes, every appearance of $v$ happens in the denominator of a fraction.  
There is then a $v'$ or a $v''$ in the numerator of each of these fractions.  
So whether $v$ is equal to $\sinh(r)$ or $d \sinh(r)$, the curvatures are still identically equal to $-1$.  
Gromov and Thurston then define $v(r) = \s(r) \sinh(r)$ where $\s(r): \R \to \R$ is a function that very slowly increases from $1$ to $d$.  
For $\s'$ and $\s''$ sufficiently small this produces a pinched negatively curved metric which is hyperbolic about a tube of the ramification locus, and which slowly warps to the full hyperbolic metric on each page of the branched cover.

The metric that we construct on the $d$-fold branched cover of $\C \H^n$ about a copy of $\C \H^{n-1}$ is similar, but the situation is now much more complicated.  
The associated metric is
	\begin{equation*}
	\mu =  h^2 \c_{n-1} + \frac{1}{4} v^2 d \th^2 + dr^2.
	\end{equation*}
where $\c_{n-1}$ denotes the complex hyperbolic metric on the core copy of $\C \H^{n-1}$.  
While terminology can vary in the literature, the above metric is {\it not} a doubly-warped product in the sense of Section 3.2.4 of \cite{Petersen}.  
The main reason for this is because the horizontal fiber is now not integrable, meaning that it is not everywhere tangent to the $r$-tube about $\C \H^{n-1}$.  
One obtains the complex hyperbolic metric $\c_n$ with the substitutions $h(r) = \cosh(r)$ and $v(r) = \sinh(2r)$ in the definition of $\mu$.  
The non-integrability of this metric leads to nonzero structure constants (equation \ref{eqn:structure constants 1} below), which create (among other issues) non-zero mixed terms for the sectional curvature tensor (Theorem \ref{thm:new curvature equations} below).

\begin{remark}\label{rmk:GT difference}
	Inserting the complex hyperbolic metric $\c_n$ about the ramification locus and then slowly increasing the angle in an identical manner as to the Gromov-Thurston metric will not result in a metric with pinched curvature and, moreover, will not generally produce a negatively curved metric.  
	This is due to the fact that there now exist terms in various components of the curvature tensor where $v$ appears without a corresponding $v'$ or $v''$.  
	If one inserts the values $c_i = 2$, $h(r) = \cosh(r)$, and $v(r) = \sigma(r) \sinh(2r) = 2 \sigma \sinh(r) \cosh(r)$, where $\s$ varies from $1$ to $d$ as above, into the formulas in Theorem \ref{thm:new curvature equations} below one has, after simplifying, the following formulas:
		\begin{itemize}
		\item $\ds{R^{\mu}_{i, 2n-1, i, 2n-1} = - \frac{\sigma' \sinh(r)}{\sigma \cosh(r)} - 1 + (\sigma^2 - 1)\frac{\sinh^2(r)}{\cosh^2(r)} \approx \s^2 - 2 }$
		\item $\ds{R^{\mu}_{i,i+1,i,i+1} = - \frac{\sinh^2(r)}{\cosh^2(r)} - \frac{4}{\cosh^2(r)} - \frac{3 \s^2 \sinh^2(r)}{\cosh^2(r)} \approx -1 - 3\s^2 }$	
		\end{itemize}
	and where the approximations are for $r$ large.  
	One sees immediately then that inserting $\s = d$ will lead to a sectional curvature tensor that is not almost $1/4$-pinched.  

\end{remark}

Note though that each of the problematic terms listed above contains a structure constant $c_i$ in Theorem \ref{thm:new curvature equations}.  
So the first part of our warping process (Theorem \ref{thm:warping construction 1} below) is to slowly ``unwind" the complex hyperbolic metric making all structure constants zero.  
This is equivalent to slowly turning the horizontal fiber of the metric until it is tangent to the $r$-tube about the core copy of $\C \H^{n-1}$.  
By turning this fiber slowly we can keep the curvatures pinched near $[-4,-1]$.  
The almost $1/4$-pinched metric in Theorem \ref{thm:main theorem} is likely not K\"{a}hler in that it is not compatible with the ambient almost-complex structure, and this is the part of the process that is the reason for this.  

The resulting metric has an integrable horizontal fiber like the hyperbolic metric, but the horizontal fiber is still a warped copy of $\c_{n-1}$ and the vertical fiber is still $\sinh(2r)$-warped.  
We call this metric the {\it integrable complex hyperbolic metric} $g_I$.
We can now increase the angle with this metric in an almost identical way as to what was described above with the Gromov-Thurston metric while keeping the curvature pinched (Theorem \ref{thm:warping construction 2}).  
This leaves us with a copy of the metric $g_I$ on each page of the branched cover.  

The final step is to ``rewind" the metric $g_I$ back to the complex hyperbolic metric $\c_n$ on each page of the branched cover (Corollary \ref{cor:warping construction 3}).  
This process is exactly backwards of the unwinding procedure discussed above.  
This results in a copy of the original complex hyperbolic metric on each page of the branched cover, which will descend to a pinched negatively curved metric on $X$.

\subsection*{Acknowledgments}
The author would first and foremost like to thank J.F. Lafont for mentioning this project to the author many years ago and for pointing out several references in the Introduction.  
The author has also worked on a related project in the past with J. Meyer and B. Tshishiku and would like to thank them for various comments during our group meetings that have been helpful with writing this paper.

\vskip 20pt

\section{Hyperbolic and complex hyperbolic metrics in polar coordinates}\label{section:polar coordinates}
In this Section we review curvature formulas for $\H^n$ written in terms of polar coordinates about a copy of $\H^{n-2}$, and we give a slightly new version of the curvature formulas for $\C \H^n$ written in polar coordinates about a copy of $\C \H^{n-1}$.  
We then prove a key Lemma (Lemma \ref{thm:pinched curvature}) which will be very important when we want to prove that our metric that ``unwinds" the complex hyperbolic metric is almost $1/4$-pinched.
The $(\H^n , \H^{n-2})$ case below is well known, but can certainly be found in \cite{Bel real} and \cite{Min warped}.  
Curvature formulas for $(\C \H^n, \C \H^{n-1})$ were derived in \cite{Bel complex} for curvatures in $[ -1, - 1/4]$, and converted to curvatures in $[-4,-1]$ in \cite{Min warped}.  
For this work we will need a slightly different version of these formulas, which we derive below.  

\subsection{The metric on $\mathbb{H}^n$ in polar coordinates about $\H^{n-2}$}\label{subsection:hn/hn-2}
Let $\H^{n-2}$ denote a totally geodesic codimension two submanifold in $\H^n$, let $r$ denote the distance to $\H^{n-2}$ within $\H^n$, and let $\y_n$, $\y_{n-2}$ denote the hyperbolic metrics on $\H^n$ and $\H^{n-2}$, respectively.  
Let $\phi_h : \H^n \to \H^{n-2}$ denote the orthogonal projection onto $\H^{n-2}$, and let $E_h(r)$ denote the $r$-tube about $\H^{n-2}$.   
Since $\H^{n-2}$ is contractible we know that topologically $E_h(r) \cong \R^{n-2} \times \S^1$.  
The metric for $\H^n$ in polar coordinates about $\H^{n-2}$ is given by
	\begin{equation}\label{eqn:h metric}
	\y_n = \cosh^2(r) \y_{n-2} + \sinh^2(r) d \th^2 + dr^2
	\end{equation}
where $d \th^2$ denotes the standard metric on the unit circle $\mathbb{S}^{1}$.  
Note that the metric in equation \eqref{eqn:h metric} is defined on $E \times (0,\infty)$, where $E$ is an arbitrary $r$-tube as defined above.  

A key feature of the real hyperbolic metric is the following.  
Let $p \in \H^{n-2}$ and let $\check{X}_1, \check{X}_2, \hdots, \check{X}_{n-2}$ be a local orthonormal frame near $p$ in $\H^{n-2}$ satisfying that $[ \check{X}_i, \check{X}_j ]_p = 0$ for all $i, j$.
Let $q \in \H^n$ be such that $\phi_h(q) = p$.
Extend the collection $( \check{X}_i )_{i=1}^{n-2}$ to vector fields $X_1, X_2, \hdots, X_{n-2}$ defined near $q$ in $\H^n$ via $d \phi_h^{-1}$, which are orthogonal to both $\ddt$ and $\ddr$.
Then the key property is that $[X_i,X_j]_q = 0$ for all $i, j$, or equivalently that the distribution determined by $(X_i)_{i=1}^{n-2}$ is integrable. 

Let $v(r)$ and $h(r)$ be positive real-valued functions of $r$.
Define the warped-product metric $\l := \l_{v,h}$ on $E \times (0,\infty)$ by
	\begin{equation*}
	\l =  h^2 \y_{n-2} + v^2 d \th^2 + dr^2.
	\end{equation*}
Of course, when $v = \sinh(r)$ and $h = \cosh(r)$ we recover the hyperbolic metric $\y_n$.
Fix vector fields $(X_i)_{i=1}^{n-2}$ as above.  
Let $X_{n-1} = \ddt$ and $X_n = \ddr$. 
Define the following orthonormal frame for $\l$:
	\begin{equation*}
	Y_i = \frac{1}{h} X_i \; \text{ for } 1 \leq i \leq n-2 \hskip 30pt  Y_{n-1} = \frac{1}{v} X_{n-1}  \hskip 30pt Y_n = X_n.
	\end{equation*}
Formulas for the components of the $(4,0)$ curvature tensor $R_\l$ for $\l$ are given by the following Theorem.

\vskip 20pt

\begin{theorem}[c.f. Section 2 of \cite{Bel real}]\label{thm:curvature formulas for h}
Let $R^{\l}_{i,j,k,l} := \l( R_\l (Y_i, Y_j)Y_k, Y_l )$.
Up to the symmetries of the curvature tensor, the only nonzero components of the (4,0) curvature tensor $R_\l$ are the following:
	\begin{align*}
	&R^\l_{i,j,i,j} = - \frac{1}{h^2} - \left( \frac{h'}{h} \right)^2 \hskip 40pt  R^\l_{i,n-1,i,n-1} = - \frac{h' v'}{hv}  \\	
	&R^\l_{i,n,i,n} = - \frac{h''}{h}   \hskip 80pt  R^\l_{n-1,n,n-1,n}= - \frac{v''}{v}
	\end{align*}
where $1 \leq i, j \leq n-2$.
\end{theorem}

\vskip 10pt

One easily checks that plugging in the values $v(r) = \sinh(r)$ and $h(r) = \cosh(r)$ gives all sectional curvatures of $-1$.

\subsection{The metric on $\C \H^n$ in polar coordinates about $\C \H^{n-1}$}\label{subsection:chn/chn-1}
Let $\C \H^{n-1}$ denote a codimension one (real codimension two) complex submanifold in $\C \H^n$, let $r$ denote the distance to $\C \H^{n-1}$ within $\C \H^n$, and let $\c_n$ and $\c_{n-1}$ denote the metrics on $\C \H^n$ and $\C \H^{n-1}$ normalized to have constant holomorphic curvature $-4$.  
Let $\phi_c : \C \H^n \to \C \H^{n-1}$ denote the orthogonal projection onto $\C \H^{n-1}$, and let $E_c(r)$ denote the $r$-tube about $\C \H^{n-1}$.   
Again, since $\C \H^{n-1}$ is contractible, topologically one has that  $E_c(r) \cong \R^{2n-2} \times \S^1$.  
The metric in $\C \H^n$ in polar coordinates about $\C \H^{n-1}$ is then given by
	\begin{equation}\label{eqn:c metric}
	\c_n = \cosh^2(r) \c_{n-1} + \frac{1}{4} \sinh^2(2r) d \th^2 + dr^2.
	\end{equation}
Note that the presence of the $1/4$ in the $d \th^2$ term is so that $\c_n$ is complete, or equivalently so that $\c_n$ has total angle of $2 \pi$ about the core copy of $\C \H^{n-1}$.  

Let $p \in \C \H^{n-1}$.  
We define a special basis $(\check{X}_1, \check{X}_2, \hdots, \check{X}_{2n-2})$ of $T_p \C \H^{n-1}$, which we call a {\it holomorphic basis} near $p$, as follows.  
We first define $\check{X}_1$ to be any unit vector in $T_p \C \H^{n-1}$.  
We then define $\check{X}_2 = J \check{X}_1$, where $J$ denotes the complex structure on $\C \H^n$.  
We call such a pair $\{ \check{X}, J \check{X} \}$ a {\it holomorphic pair}.
Let $\check{X}_3$ be any unit vector in $T_p \C \H^{n-1}$ which is orthogonal to $\text{span}(\check{X}_1, \check{X}_2)$, and let $\check{X}_4 = J \check{X}_3$.  
We continue in this way to construct the entire orthonormal basis $(\check{X}_1, \check{X}_2, \hdots, \check{X}_{2n-2})$ of $T_p \C \H^{n-1}$ in such a way that, for $i$ odd, the pair $(\check{X}_i, \check{X}_{i+1})$ is a holomorphic pair.
Via a standard construction, we can extend this basis to a neighborhood of $p$ in $\C \H^{n-1}$ in such a way that $[ \check{X}_i, \check{X}_j ]_p = 0$ for all $i$ and $j$.  
Let us note that if $i$ is an odd integer then $\text{exp}_p(\text{span}(\check{X}_i, \check{X}_{i+1}))$ is a complex line, whereas if $\{ \check{X}_i, \check{X}_j \}$ do not form a holomorphic pair then $\text{exp}_p(\text{span}(\check{X}_i, \check{X}_j))$ is a totally real totally geodesic subspace of $\C \H^{n-1}$.   

Let $q \in \C \H^n$ be such that $\phi_c(q) = p$.
Extend the collection $( \check{X}_i )_{i=1}^{2n-2}$ to vector fields $X_1, X_2, \hdots, X_{2n-2}$ defined near $q$ in $\C \H^n$ via $d \phi_c^{-1}$ in such a way that they are invariant with respect to both $\theta$ and $r$.  
We will call such a frame a {\it holomorphic frame} near $q$.  
Just as above, we need to understand the Lie brackets of this frame associated to the metric in \eqref{eqn:c metric}. 
It is proved in \cite{Bel complex} that there exist {\it structure constants} $c_{i}$ such that
	\begin{equation}\label{eqn:structure constants 1}
	[X_i,X_{i+1}] = c_i \ddt \qquad \text{for } i \text{ odd} 
	\end{equation}
and that $[X_i,X_j] = 0$ whenever $\{ X_i, X_j \}$ is not a holomorphic pair.  
Moreover, in \cite{Bel complex} it is actually proved\footnote{Belegradek obtains structure constants of $\pm \dfrac{1}{2}$, but the structure constants here are $\pm 2$ due to the metric being scaled to have sectional curvatures in $[-4,-1]$ instead of $\ds{ \left[ -1, \frac{-1}{4} \right] }$}  that
	\begin{equation}\label{eqn:structure constants 2}
	c_i = \pm 2 \qquad \forall i.  
	\end{equation}
In our curvature calculations later in the paper, we will need to understand the complex hyperbolic metric with arbitrary structure constants.  
This is why we cannot use the curvature calculations in \cite{Bel complex} and why we keep $c_i$ in equation \eqref{eqn:structure constants 1}.

Let $v(r)$ and $h(r)$ be positive real-valued functions of $r$.
Define the metric $\mu := \mu_{v,h}$ on $E \times (0,\infty)$ by
	\begin{equation}\label{eqn:mu}
	\mu =  h^2 \c_{n-1} + \frac{1}{4} v^2 d \th^2 + dr^2.
	\end{equation}
Note that when $v = \sinh(2r)$ and $h = \cosh(r)$ we recover the complex hyperbolic metric $\c_n$.
Fix a holomorphic frame $(X_i)_{i=1}^{2n-2}$ as above.  
Let $X_{2n-1} = \ddt$ and $X_{2n} = \ddr$. 
Define the following orthonormal frame for $\mu$:
	\begin{equation}
	Y_i = \frac{1}{h} X_i \; \text{ for } 1 \leq i \leq 2n-2 \hskip 30pt  Y_{2n-1} = \frac{1}{\frac{1}{2} v} X_{2n-1}  \hskip 30pt Y_{2n} = X_{2n}.  \label{eqn:ON basis}
	\end{equation}
Formulas for the components of the $(4,0)$ curvature tensor $R_\mu$ of $\mu$ are given by the following Theorem.

\vskip 20pt

\begin{theorem}[compare Sections 7 and 8 of \cite{Bel complex} with Theorem 4.3 of \cite{Min warped}]\label{thm:new curvature equations}
Let $R^{\mu}_{i,j,k,l} := \mu( R_\mu (Y_i, Y_j)Y_k, Y_l )$, and let $(X_i)_{i=1}^{2n-2}$ be a holomorphic basis.
Then, up to the symmetries of the curvature tensor, we have the following formulas for the components of the (4,0) curvature tensor $R_\mu$ with respect to the corresponding orthonormal basis $(Y_i)_{i=1}^{2n}$:
	\begin{align*}
	&R^{\mu}_{i,j,i,j} = - \frac{1}{h^2} - \left( \frac{h'}{h} \right)^2  \hskip 40pt  R^{\mu}_{i,2n-1,i,2n-1} = - \frac{h' v'}{hv} + \frac{c_i^2 v^2}{16 h^4}  \\	
	&R^{\mu}_{i, i+1, i, i+1} = - \left( \frac{h'}{h} \right)^2 - \frac{4}{h^2} - \frac{3 c_i^2 v^2}{16 h^4} \\
	&R^{\mu}_{i,2n,i,2n} = - \frac{h''}{h}   \hskip 70pt  R^{\mu}_{2n-1,2n,2n-1,2n}= - \frac{v''}{v}  \\
	&R^{\mu}_{i, i+1, 2n-1, 2n} = 2 R^{\mu}_{i, 2n-1, i+1, 2n} = -2R^{\mu}_{i, 2n, i+1, 2n-1} = - c_{i} \frac{v}{2h^2} \left( \ln \frac{v}{h} \right)'  \\
	&R^{\mu}_{i, i+1, k, k+1} = 2R^{\mu}_{i, k, i+1, k+1} = -2R^{\mu}_{i, k+1, i+1, k} = - \frac{2}{h^2} - \frac{c_i c_k v^2}{8h^4}
	\end{align*}
where $1 \leq i, j, k \leq 2n-2$, $k$ is an odd integer different from $i$, and $j \neq i, i+1$.  
Also, any equations using both $i$ and $i+1$ assume that $i$ is an odd integer.
\end{theorem}

\vskip 10pt

\begin{proof}
The proof here is analogous to Section 6 of \cite{Bel complex}.  
We compute the $A$ and $T$ tensors of the metric
	\begin{equation*}
	\mu_r = h^2 \c_{n-1} + \frac{1}{4} v^2 d \th^2.
	\end{equation*}
We then use Theorem 9.28 of \cite{Besse} to compute formulas for the components of the sectional curvature tensor $R_{\mu_r}$ of $\mu_r$.
Finally, we use the equations in Appendix B of \cite{Bel complex} to obtain the components of $R_{\mu}$.  

First, by identical reasoning to \cite{Bel complex}, the fibers of the Riemannian submersion are totally geodesic.  
Thus, the $T$ tensor is identically zero.  
We now compute the $A$ tensor.  
For this, we fix the notation that $i$ is an odd integer and $j \neq i, i+1$.  

By (\cite{Besse} eqn.\ 9.24) we have that
	\begin{equation*}
	A_{X_i}X_{i+1} = \frac{1}{2} \mathscr{V} [X_i, X_{i+1}] = \frac{c_i}{2} X_{2n-1}	\hskip 20pt	\text{and}	\hskip 20pt	A_{X_i}X_j = 0.
	\end{equation*}
Then by (\cite{Besse} eqn.\ 9.21d), we have
	\begin{align*}
	\mu_r (A_{X_i} X_{2n-1}, X_{i+1} ) &= - \mu_r ( A_{X_i} X_{i+1}, X_{2n-1} )  \\
	&= - \frac{c_i}{2} \mu_r ( X_{2n-1}, X_{2n-1} ) = - \frac{c_i}{8} v^2.
	\end{align*}
In a similar manner we see that all other components of $A_{X_i}X_{2n-1}$ are zero.  
Thus, we have that
	\begin{equation*}
	A_{X_i}X_{2n-1} = - \frac{c_i v^2}{8 h^2} X_{i+1}	\hskip 20pt	\text{and}	\hskip 20pt	A_{X_{i+1}}X_{2n-1} = \frac{c_i v^2}{8 h^2} X_i.
	\end{equation*}
	
We are now ready to compute the components of $R_{\mu_r}$ in terms of the basis $(Y_i)_{i=1}^{2n-1}$.  
In the following calculations we use Theorem 9.28 from \cite{Besse}, as well as the linearity of the curvature tensor.  
Also, in the fourth bullet point, $j$ denotes an odd integer different from $i$.
	\begin{align*}
	\bullet \; \mu_r ( &R_{\mu_r}(X_i, X_{2n-1})X_i, X_{2n-1} ) = \mu_r ( A_{X_i}X_{2n-1}, A_{X_i}X_{2n-1} ) = \frac{c_i^2 v^4}{64 h^4} \cdot h^2 = \frac{c_i^2 v^4}{64 h^2}  \\
	& \implies \; \mu_r ( R_{\mu_r}(Y_i, Y_{2n-1})Y_i, Y_{2n-1} ) = \frac{4}{h^2 v^2} \cdot \frac{c_i^2 v^4}{64 h^2} = \frac{c_i^2 v^2}{16 h^4}.  \\
	\bullet \; \mu_r ( &R_{\mu_r}(X_i, X_{i+1})X_i, X_{i+1} ) = \mu_r ( \check{R}_{\mu_r} (\check{X}_i, \check{X}_{i+1})\check{X}_i, \check{X}_{i+1} ) - 3 \mu_r ( A_{X_i}X_{i+1}, A_{X_i}X_{i+1} )  \\
	&= -4h^2 - \frac{3 c_i^2 v^2}{16}  \\
	& \implies \; \mu_r ( R_{\mu_r}(Y_i, Y_{i+1})Y_i, Y_{i+1} ) = \frac{1}{h^4} \cdot \left( -4h^2 - \frac{3 c_i^2 v^2}{16} \right) = - \frac{4}{h^2} - \frac{3 c_i^2 v^2}{16 h^4}.  \\
	\bullet \; \mu_r ( &R_{\mu_r}(X_i, X_j)X_i, X_j ) = \mu_r ( \check{R}_{\mu_r} (\check{X}_i, \check{X}_j)\check{X}_i, \check{X}_j ) - 3 \mu_r ( A_{X_i}X_j, A_{X_i}X_j ) = -h^2   \\
	& \implies \; \mu_r ( R_{\mu_r}(Y_i, Y_j)Y_i, Y_j ) = \frac{1}{h^4} \cdot (-h^2) = - \frac{1}{h^2}.  \\
	\bullet \; \mu_r ( &R_{\mu_r}(X_i, X_{i+1})X_j, X_{j+1} ) = \mu_r ( \check{R}_{\mu_r} (\check{X}_i, \check{X}_{i+1})\check{X}_j, \check{X}_{j+1} ) - 2 \mu_r ( A_{X_i}X_{i+1}, A_{X_j}X_{j+1} )   \\
	&= -2h^2 - 2 \left( \frac{c_i c_j}{4} \cdot \frac{v^2}{4} \right) = -2h^2 - \frac{c_i c_j v^2}{8}  \\
	& \implies \; \mu_r ( R_{\mu_r}(Y_i, Y_{i+1})Y_j, Y_{j+1} ) = \frac{1}{h^4} \cdot \left( -2h^2 - \frac{c_i c_j v^2}{8} \right) = - \frac{2}{h^2} - \frac{c_i c_j v^2}{8h^4}.
	\end{align*}
	
A few quick remarks about the above calculations.  
All of the curvatures with a ``hat" above the symbols occur in the horizontal fiber, which is isometric to the $h^2$-multiple of $\c_{n-1}$.  
One can use Proposition IX.7.3 from \cite{KN} to verify that the corresponding curvatures above in $\c_{n-1}$ are indeed $-4$, $-1$, and $-2$.  
One can also deduce from Theorem 9.28 in \cite{Besse} that $R^{\mu_r}_{i+1, 2n-1, i+1, 2n-1} = R^{\mu_r}_{i, 2n-1, i, 2n-1}$, and that $R^{\mu_r}_{i, i+1, j, j+1} = 2 R^{\mu_r}_{i, j, i+1, j+1} = -2 R^{\mu_r}_{i, j+1, i_1, j}$ (and recall that this subscript notation is just the components of the curvature tensor with respect to the ON basis $(Y_i)$).

Combining the above computations with the equations in Appendix B of \cite{Bel complex} proves the Theorem.

\end{proof}

\subsection{The Pinching Lemma}
The following Lemma is our main technical result which ensures that the metric constructed below in Theorem \ref{thm:warping construction 1} and Corollary \ref{cor:warping construction 3} is almost negatively $1/4$-pinched.  
In the statement below, $K_g(\sigma)$ denotes the sectional curvature of the $2$-plane $\sigma$ with respect to the metric $g$.  

\vskip 20pt

\begin{lemma}\label{thm:pinched curvature}
Let $g$ denote the metric $\mu$ from equation \eqref{eqn:mu} on $X = \R^{n-2} \times \S^1 \times (0,\infty)$ with $h(r) = \cosh(r)$ and $v(r) = \sinh(2r) = 2 \sinh(r) \cosh(r)$.  
Let $q = (p, \theta, r) \in X$.  
Then for all $\e > 0$, there exists $R > 0$ large and $\eta > 0$ small such that for all 2-planes $\sigma \subseteq T_q (X)$ where $r > R$, we have that $K_g(\s) \in (-4-\e, -1+\e)$ provided that all structure constants $c_i$ from equation \eqref{eqn:structure constants 1} satisfy $c_i \in (-2 - \eta, 2 + \eta)$.  
\end{lemma}

\vskip 10pt

\begin{proof}
Let $\s \subseteq T_q(X)$ be a 2-plane.  
Note that, for almost all such 2-planes $\s$, we have that $d \phi_c (\s)$ has dimension 2 in the base $T_p \C \H^{n-1}$.  
We will prove Lemma \ref{thm:pinched curvature} in the case when $d \phi_c (\s)$ has dimension 2, and the result then extends to all 2-planes by continuity.  
Also, we will prove the statement for $c_i \in [-2,2]$.  
The result will also then extend to $(-2-\eta, 2+ \eta)$ for some $\eta > 0$ by continuity.  

In a similar way as in Section 9 of \cite{Bel complex}, we will choose our holomorphic frame near $q$ depending on the position of $\s$.  
Since $d \phi_c (\s)$ has dimension 2, $\s$ must contain a unit vector $B$ that is tangent at $q$ to the $r$-tube about $\C \H^{n-1}$.  
Choose $\check{X}_1$ to be a unit-vector in $T_p \C \H^{n-1}$ that is parallel to $d \phi_c (B)$.  
Let $\check{X}_2 = J \check{X}_1$.  
Finally, let $A$ be a unit vector in $T_q X$ perpendicular to $B$ such that $\s = \text{span}(A,B)$, and choose a unit vector $\check{X}_3 \in T_p \C \H^{n-1}$ perpendicular to both $\check{X}_1$ and $\check{X}_2$ so that $d \phi_c (A) \in \text{span}(\check{X}_1, \check{X}_2, \check{X}_3)$.  
We then extend $\{ \check{X}_1, \check{X}_2, \check{X}_3 \}$ to a holomorphic frame $\{ \check{X}_i \}_{i=1}^{2n-2}$ near $p$ as in Subsection \ref{subsection:chn/chn-1}, and pull this back via $d \phi^{-1}$ to a holomporphic frame $\{ X_i \}$ near $q$.  
Let $\ds{ X_{2n-1} = \ddt }$ and $\ds{ X_{2n} = \ddr }$, and define the corresponding orthonormal basis $(Y_i)$ as in equation \eqref{eqn:ON basis}.  
Note that we can write
	\begin{equation*}
	 A = a_1 Y_1 + a_2 Y_2 + a_3 Y_3 + a_4 Y_{2n-1} + a_5 Y_{2n} \qquad B = b_1 Y_1 + b_4 Y_{2n-1}
	 \end{equation*}
where 
	\begin{equation}\label{eqn:coefficient conditions}
	a_1^2 + a_2^2 + a_3^2 + a_4^2 + a_5^2 = 1 = b_1^2 + b_4^2 \quad \text{ and } \quad a_1 b_1 + a_4 b_4 = 0.
	\end{equation}  

Letting $h(r) = \cosh(r)$ and $v(r) = \sinh(2r) = 2 \sinh(r) \cosh(r)$ in Theorem \ref{thm:new curvature equations}, and using the subscripts for the basis chosen above gives the following formulas for the components of the curvature tensor $R^{\mu} := R$.  
	\begin{align*}
	&R_{1313} = R_{2323} =  -1 \hskip 40pt  R_{1,2n,1,2n} = R_{2,2n,2,2n} = R_{3,2n,3,2n} = -1   \\	
	&R_{1212} = - \left( \frac{\sinh(r)}{\cosh(r)} \right)^2 - \frac{4}{\cosh^2(r)} - \frac{3 c_1^2 \sinh^2(r)}{4 \cosh^2(r)} \approx -1 - \frac{3 c_1^2}{4} \\
	&R_{1,2n-1,1,2n-1} = R_{2,2n-1,2,2n-1} = -1  - \frac{\sinh^2(r)}{\cosh^2(r)} + \frac{c_1^2 \sinh^2(r)}{4 \cosh^2(r)} \approx -2 + \frac{c_1^2}{4}   \\
	&R_{3,2n-1,3,2n-1} \approx -2 + \frac{c_3^2}{4}  \hskip 70pt  R_{2n-1,2n,2n-1,2n}= - 4  \\
	&R_{1, 2, 2n-1, 2n} = 2 R_{1, 2n-1, 2, 2n} = -2R_{1, 2n, 2, 2n-1} = - c_{1}  
	\end{align*}
where ``$\approx"$ is for $r$ sufficiently large.  

The sectional curvature $K(\s)$ is $\s$ is then computed by
	\begin{align*}
	K(\s) =& g( R(A,B)A, B )  \\
	=& b_1^2 g( R (A, Y_1)A, Y_1 ) + b_4^2 g( R (A, Y_{2n-1})A, Y_{2n-1} ) + 2 b_1 b_4 g( R (A, Y_1)A, Y_{2n-1} ) \\
	=& b_1^2 \left( a_2^2 R_{1212} + a_3^2R_{1313} + a_4^2 R_{1,2n-1,1,2n-1} + a_5^2 R_{1,2n,1,2n} \right)  \\
	&+b_4^2 \left( a_1^2 R_{1,2n-1,1,2n-1} + a_2^2 R_{2,2n-1,2,2n-1} + a_3^2 R_{3,2n-1,3,2n-1} + a_5^2 R_{2n-1,2n,2n-1,2n} \right)  \\
	&+2b_1 b_4 \left(a_2 a_5 R_{1,2,2n-1,2n} - a_1 a_4 R_{1,2n-1,1,2n-1} - a_2 a_5 R_{1,2n,2,2n-1} \right)
	\end{align*}
Plugging in the above curvature formulas gives
	\begin{align*}
	K(\s) \approx & b_1^2 \left( a_2^2 \left( -1 - \frac{3c_1^2}{4} \right) - a_3^2 + a_4^2 \left( -2 + \frac{c_1^2}{4} \right) - a_5^2 \right)  \\
	&+ b_4^2 \left( a_1^2 \left( -2 + \frac{c_1^2}{4} \right) + a_2^2 \left( -2 + \frac{c_1^2}{4} \right) + a_3^2 \left( -2 + \frac{c_3^2}{4} \right) -4a_5^2 \right)  \\
	&+ 2b_1 b_4 \left( -c_1 a_2 a_5 - a_1 a_4 \left( -2 + \frac{c_1^2}{4} \right) - a_2 a_5 \cdot \frac{c_1}{2} \right).
	\end{align*}
Rearranging terms and applying the fact that $b_1^2 + b_4^2 = 1$ yields
	\begin{align*}
	K(\s) \approx & -a_2^2 -a_3^2 - a_5^2 -3a_5^2 b_4^2 - (a_1 b_4 - a_4 b_1)^2 - a_2^2 b_1^2 \left( \frac{3c_1^2}{4} \right)  \\
	&+ \left( -1 + \frac{c_1^2}{4} \right) \left[ \left( a_1 b_4 + a_4 b_1 \right)^2 + a_2^2 b_4^2 \right] - 3 c_1 a_2 a_5 b_1 b_4 +a_3^2 b_4^2 \left( -1 + \frac{c_3^2}{4} \right).
	\end{align*}
Now, notice that since $0 = a_1 b_1 + a_4 b_4$, we have
	\begin{equation*}
	(a_1b_4 - a_4 b_1)^2 = (a_1b_4 - a_4 b_1)^2 + (a_1 b_1 + a_4 b_4)^2 = a_1^2 b_4^2 + a_4^2 b_1^2 + a_1^2 b_1^2 + a_4^2 b_4^2 = a_1^2 + a_4^2.
	\end{equation*}
Substituting this above and using that $a_1^2 + \hdots + a_5^2 = 1$ gives
	\begin{align*}
	K(\s) \approx & -1 - 3 \left( a_5^2 b_4^2 + c_1 a_2 a_5 b_1 b_4 + \frac{1}{4} c_1^2 a_2^2 b_1^2 \right) + \left( -1 + \frac{c_1^2}{4} \right) \left[ \left( a_1 b_4 + a_4 b_1 \right)^2 + a_2^2 b_4^2 \right]  \\ \nonumber
	&+a_3^2 b_4^2 \left( -1 + \frac{c_3^2}{4} \right)  \\ \nonumber
	= & -1 - 3 \left( a_5 b_4 + \frac{1}{2} c_1 a_2 b_1 \right)^2 + \left( -1 + \frac{c_1^2}{4} \right) \left[ \left( a_1 b_4 + a_4 b_1 \right)^2 + a_2^2 b_4^2 \right]   \\ \nonumber
	&+a_3^2 b_4^2 \left( -1 + \frac{c_3^2}{4} \right).  \label{eqn:curvature condition}
	\end{align*}
This last expression is clearly negative.  
One can check that, under the conditions in equation \eqref{eqn:coefficient conditions} and if $|c_i| \leq 2$ for $i=1,3$, this last expression
	\begin{itemize}
	\item  is maximized when $c_1 = \pm 2$, $c_3 = \pm 2$, and $\ds{ a_5 b_4 + \frac{1}{2} c_1 a_2 b_1 = 0 }$.  In this case, we have that $K(\s) \approx -1$.  
	\item  is minimized when $c_1 = c_3 = 0$, $b_4 = \pm 1$, and $a_5 = \pm 1$.  Note that we then must have all other $a$'s and $b$'s equal to $0$.  This then yields $K(\s) \approx -4$.  
	\end{itemize}
	
\end{proof}

\begin{remark}
Note that the condition ``for $r$ sufficiently large" in Lemma \ref{thm:pinched curvature} is really only used in the proof above in order to simplify calculations.  
This is also the only situation where we will need to use Lemma \ref{thm:pinched curvature}.  
But an inspection of the formulas above shows that it is possible that this Theorem holds for all positive values of $r$ and, at the very least, the value of $R$ in the statement of the Theorem need not be very large.  
\end{remark}


\vskip 20pt

\section{Two special metrics and almost pinched metrics that interpolate between them}\label{section:2 metrics}

Fix $E = \R^{2n-2} \times \S^1$.  
In this Section we define two metrics on the product $E \times (0, \infty)$ and prove that we can interpolate between various metrics in an almost $1/4$-pinched way as discussed in the Introduction (see Theorem \ref{thm:warping construction 1}, Corollary \ref{cor:warping construction 3}, and Theorem \ref{thm:warping construction 2} below).  

\vskip 20pt

\subsection{The integrable complex hyperbolic metric $g_{I}$}
The {\it integrable complex hyperbolic metric} $g_{I}$ is defined by setting $h(r) = \cosh(r)$ and $v(r) = \sinh(2r) = 2 \sinh(r) \cosh(r)$ in equation \eqref{eqn:mu}, and setting all structure constants defined in \eqref{eqn:structure constants 1} identically equal to zero.  
This is the metric one would obtain if the complex hyperbolic metric were integrable, that is, if the horizontal fiber $d \phi_c^{-1} (T_p \C \H^{n-1})$ were always tangent to $E_c(r)$ and orthogonal to $\phi_c^{-1}(p)$ as is the case with the hyperbolic metric.  
So, as a metric, we have that
	\begin{equation*}
	g_{I} = \cosh^2(r) \c_{n-1} + \frac{1}{4} \sinh^2(2r) d \th^2 + dr^2
	\end{equation*}
but with different structure constants than the metric $\c_n$.  

Setting $c_i = c_k = 0$ in Theorem \ref{thm:new curvature equations}, as well as inserting $h= \cosh(r)$ and $v=2 \sinh(r) \cosh(r)$, proves the following Lemma.

\vskip 20pt

\begin{lemma}\label{lemma:curvature formulas for gI}
Let $R$ denote the sectional curvature tensor with respect to the integrable complex hyperbolic metric $g_I$.  
Then, using the notation of Theorem \ref{thm:new curvature equations}, we have:
	\begin{align*}
	&R_{i,j,i,j} = - 1   \hskip 40pt  R_{i,2n-1,i,2n-1} = - 1 - \tanh^2(r)	\hskip 40pt	R_{i,2n,i,2n} = - 1  \\	
	&R_{i,i+1,i,i+1} = -1 - \frac{3}{\cosh^2(r)}	\hskip 40pt	R_{2n-1,2n,2n-1,2n}= - 4  \\
	&R_{i, i+1, 2n-1, 2n} = R_{i, 2n-1, i+1, 2n} = R_{i, 2n, i+1, 2n-1} = 0  \\
	&R_{i, i+1, k, k+1} = 2R_{i, k, i+1, k+1} = -2R_{i, k+1, i+1, k} = - \frac{2}{\cosh^2(r)}
	\end{align*}
where $1 \leq i, j, k \leq 2n-2$, $k$ is an odd integer different from $i$, and $j \neq i, i+1$.
Also, any equations using both $i$ and $i+1$ assume that $i$ is an odd integer.
\end{lemma}

\vskip 10pt

By Lemma \ref{thm:pinched curvature} we know that, given $\e > 0$, all sectional curvatures of $g_I$ lie in $(-4-\e, -1+\e)$ for $r$ sufficiently large.  

The next Theorem proves that there exists an almost $1/4$-pinched metric that interpolates between the complex hyperbolic metric $\c_n$ and the metric $g_I$.  
This Theorem and the following Corollary provide two of the three steps in the construction of the almost $1/4$-pinched Riemannian metric in Theorem \ref{thm:main theorem}.

\vskip 20pt

\begin{theorem}\label{thm:warping construction 1}
Let $\e > 0$.  
There exist positive constants $r_1 < r_2$ and a Riemannian metric $g$ defined on $E \times (0, \infty)$ such that
	\begin{enumerate}
	\item  $g$ restricted to $E \times (0, r_1)$ is the metric $\c_n$.  
	\item  $g$ restricted to $E \times (r_2, \infty)$ is the metric $g_I$.  
	\item  all sectional curvatures of $g$ lie in the interval $(-4-\e, -1 + \e)$.  
	\end{enumerate}
\end{theorem}

\vskip 10pt

\begin{proof}  Morally, to construct this metric one just, very slowly, turns the horizontal fiber of the complex hyperbolic metric until it is tangent to the $r$-tube about a copy of $\C \H^{n-1}$ in such a way that all structure constants decrease from $2$ to $0$.  
Since the structure constants are, in fact, constant, this can be done in a manner which is invariant of both $\theta$ and the location of the base point in $\C \H^{n-1}$.  
By moving this fiber sufficiently slowly one can locally $C^2$-approximate this metric by a metric where $c_i$ is constant for each $i$.  
Lemma \ref{thm:pinched curvature} then shows that the curvature is almost $1/4$-pinched.  
The argument below gives technical details on one way that this can all be done, but the author feels that the above description is more enlightening.

From Section 2 we know that both $\C \H^n / \C \H^{n-1}$ and $\H^{2n} / \H^{2n-2}$ are diffeomorphic to $E \times (0, \infty)$ where, for any $p \in \R^{2n-2} \cong \C \H^n \cong \H^{2n}$ and $r > 0$, we have that $(p, \S^1, r)$ corresponds to the intersection $E_h(r) \cap \phi^{-1}_h (p)$ or $E_c(r) \cap \phi^{-1}_c (p)$.  
So, by fixing diffeomorphisms $f_h: \H^{2n-2} \to \R^{2n-2}$ and $f_c :\C \H^{n-1} \to \R^{2n-2}$, we can naturally identify both $\C \H^n / \C \H^{n-1}$ and $\H^{2n} / \H^{2n-2}$ with $E \times (0, \infty)$.  


Via the diffeomorphisms $f_h$ and $f_c$ we will henceforth abuse notation and apply all functions and metrics from Section 2 to the space $E \times (0, \infty)$.  
For example, the metric $\c_n$ will really mean $(f_c^{-1})^* \c_n $, and so on.  

Let $p \in \C \H^{n-1}$ be arbitrary, and let $( \check{X}_1, \check{X}_2, \hdots, \check{X}_{2n-2})$ be a holomorphic basis at $p$.  
Extend this basis to a local frame near $p$ which satisfies that $[\check{X}_i, \check{X}_j]_p = 0$ for all $i$ and $j$.  
Let $U$ be a neighborhood of $p$ in $\C \H^{n-1}$ on which this frame is defined.  

Via the orthogonal projections $\phi_h$ and $\phi_c$ we will define two sets of vector fields on the space $U \times \S^1 \times (0, \infty)$.  
Using $\phi_h^{-1}$ we pull the frame $( \check{X}_i )$ back to obtain the frame $(X^h_i)$, and using $\phi_c^{-1}$ we similarly obtain the frame $(X^c_i)$. 

Define a smooth function $\alpha: (0, \infty) \to [0,1]$ which satisfies
	\begin{itemize}
	\item  $\a(r) = 1$ for all $r \leq r_1$.  
	\item  $\a(r) = 0$ for all $r \geq r_2$.  
	\item  over the interval $(r_1,r_2)$ we have $\a'(r) < \d$ and $\a''(r) < \d$ for some prescribed $\d > 0$ which depends on $\e$.  
	\end{itemize}
For any given $\d > 0$, such a function $\a$ clearly exists provided that the distance $r_2 - r_1$ is sufficiently large.  

For $1 \leq i \leq 2n-2$ we define the vector field $X_i$ on $U \times \S^1 \times (0, \infty)$ by
	$$ X_i = \a X_i^c + (1-\a) X_i^h. $$
Note that, since $\a$ is only a function of $r$ and since $X_i^c$ and $X_i^h$ are both defined via the same vector field $\check{X}_i$, this ``unbending" process is independent of $p$ and $\theta$.  
We define the metric $g$ to be
	\begin{equation}\label{eqn:bending metric}
	g = \cosh^2(r) \left( dX_1^2 + dX_2^2 + \hdots + dX_{2n-2}^2 \right) + \frac{1}{4} \sinh^2(2r) d \theta^2 + dr^2 
	\end{equation}
where $dX_i$ is the covector field dual to $X_i$ for all $i$.  
Note that the horizontal fiber of $g$ deforms from $\text{span}(X_1^c, X_2^c, \hdots, X_{2n-2}^c)$ to $\text{span}(X_1^h, X_2^h, \hdots, X_{2n-2}^h)$ in a manner that is independent of the choice of our original frame since $\a$ only depends on $r$.  
Thus, this metric is really well-defined on all of $E \times (0, \infty)$ .  
Also, note that $g = \c_n$ when $r < r_1$ and $g = g_I$ for $r > r_2$.  
   
To show that the metric $g$ is almost $1/4$-pinched, we need to compute the Lie brackets of the basis $(X_i)$ of the horizontal fiber.
We will show that, for $\d$ sufficiently small and $r_1$ and $r_2-r_1$ sufficiently large, we can $C^2$-approximate $g$ by a metric which satisfies Lemma \ref{thm:pinched curvature}.
Then, since sufficiently $C^2$-small changes in a Riemannian metric cause arbitrarily small changes to its sectional curvature, the proof will be complete.  

To start, note that all vector fields are invariant with respect to $\th$, and therefore $[ X_i, \ddt ] = 0$ for all $i$.  
Next, consider all pairs $(\check{X}_i, \check{X}_j)$ that are not a holomorphic pair (that is, where $\{ i, j \} \neq \{ k, k+1 \}$ for some odd $k$).  
These vector fields were chosen independently of each other, and so $[X^h_i, X^c_j] = 0 = [X^c_i, X^h_j]$.  
From this it follows that $[X_i, X_j] = 0$.  

We now consider a holomorphic pair $(\check{X}_i, \check{X}_{i+1} )$.  
In what follows we assume that the structure constant for this pair is $+2$ instead of $-2$.  
If the sign is reversed, then we just swap the order of the indices.  
We have
	\begin{align*}
	[X_i, & X_{i+1}] = [\a X_i^c + (1-\a) X_i^h, \a X_{i+1}^c + (1-\a) X_{i+1}^h ]  \\
	&= \a^2 [X_i^c, X_{i+1}^c] + \a (1 - \a) \left( [X_i^c, X_{i+1}^h] + [X_i^h, X_{i+1}^c] \right) + (1-\a)^2 [X_i^h, X_{i+1}^h]  \\
	&= 2 \a^2 \ddt + \a (1 - \a) \left( [X_i^c, X_{i+1}^h] + [X_i^h, X_{i+1}^c] \right)
	\end{align*}
	
To analyze the mixed Lie brackets in this equation, write $X_i^h = \sum_{j=1}^{2n-1} a_{i,j} X_j^c$ for some locally defined coefficient functions $a_{i,j}$ and where $X_{2n-1}^c = \ddt$.
When we plug this sum into the Lie bracket $[X_i^h, X_{i+1}^c]$ and evaluate at $q = (p, \th, r)$, the only term that does not zero out is $[a_{i,i} X_i^c, X_{i+1}^c ]$ since we chose our frame so that $[\check{X}_i, \check{X}_{i+1}]_p = 0$.  
Therefore, 
	$$[X_i^h, X_{i+1}^c] = a_{i,i} [X_i^c, X_{i+1}^c] = 2 a_{i,i} \ddt .$$  
Note that, for $r$ large, we have
	\begin{align*}
	\frac{1}{4} e^{2r} &\approx \y_{2n}(X_i^h, X_i^h) \geq \y_{2n} (a_{i,i+1} X_{i+1}^c, a_{i, i+1} X_{i+1}^c) \\
	&= (a_{i, i+1})^2 \y_{2n} (X_{i+1}^c, X_{i+1}^c ) \approx (a_{i, i+1})^2 \frac{1}{4} e^{2r}. 
	\end{align*}
and therefore $|a_{i,i}| \leq 1 + \gamma$ for some $\gamma > 0$ small. 
Analogously, this all also holds for the Lie bracket $[X_i^c, X_{i+1}^h]$ where we write $X_{i+1}^h = \sum_{j=1}^{2n-1} b_{i+1,j} X_j^c$.  
We therefore have
	\begin{align*}
	\bigr| [X_i,  X_{i+1}] \bigr| &= \biggr| 2 \a^2 \ddt + \a (1 - \a) \left( 2 b_{i+1, i} \ddt + 2 a_{i, i+1} \ddt \right) \biggr|  \\
	&\leq |2 \a^2 + 4 \a (1-\a) + \eta | \biggr| \ddt \biggr|
	\end{align*}
where $\eta$ is a small error term due to $\gamma$ above that can be made arbitrarily small by choosing $r$ sufficiently large.
  
Note that, ignoring $\eta$, the coefficient within the absolute values on the right-hand side of the inequality reduces to $4\a - 2 \a^2$.  
If we let $b(r) = 4 \a(r) - 2 \a^2 (r)$, we see that $b(r_1) = 2$, $b(r_2) = 0$, and $b$ is decreasing over $[r_1, r_2]$.  
Thus, the metric in equation \eqref{eqn:bending metric} satisfies the condition on the structure constants for Lemma \ref{thm:pinched curvature}.
One can check that, if the values for both $a_{i, i+1}$ and $b_{i+1, i}$ are both $-1$ (or any values for each between $-1$ and $1$), this coefficient still falls in the interval $(-2-\eta ,2+\eta )$ for sufficiently large $r$ and sufficiently small $\g$.  

The last Lie brackets that we need to consider are the Lie brackets of the form $[X_i, \partial r]$ where $\partial r = \ddr$.  
This is easier to consider with respect to an orthonormal basis $(Y_i)$ where $Y_i = \frac{1}{\cosh(r)} X_i$ for all $1 \leq i \leq 2n-2$.  
We then compute
	\begin{align*}
	[Y_i, \partial r] &= - \left[ \partial r, \frac{\a}{\cosh(r)} X_i^c + \frac{1-\a}{\cosh(r)} X_i^h \right]  \\
	&= - \left( \frac{\a'}{\cosh(r)} - \frac{\a \sinh(r)}{\cosh^2(r)} \right) X_i^c - \left( \frac{-\a'}{\cosh(r)} - \frac{(1-\a)\sinh(r)}{\cosh^2(r)} \right) X_i^h  \\
	&= \frac{\a'}{\cosh(r)} (X_i^h - X_i^c) + \tanh(r) Y_i.
	\end{align*}
With respect to both the hyperbolic and complex hyperbolic metric one has that $[Y_i, \partial r] = \tanh(r) Y_i$.  
So the ``error term" in the above summand is $\frac{\a'}{\cosh(r)} (X_i^h - X_i^c)$.  
If this term were equal to $0$, then Lemma \ref{thm:pinched curvature} would apply and our proof would be complete.  
But the vectors $X_i^h$ and $X_i^c$ grow at a rate of approximately $\cosh(r)$ with respect to all of the metrics here.  
So this error term can be made as small as possible by requiring that $\d$ be chosen sufficiently small.  
We can then $C^2$-approximate this metric by one which satisfies Lemma \ref{thm:pinched curvature} to complete the proof.  
\end{proof}

\vskip 20pt

\begin{cor}\label{cor:warping construction 3}
Let $\e > 0$.  
Then there exist positive constants $r_3 < r_4$ and a Riemannian metric $g$ defined on $E \times [r_3, \infty)$ such that
	\begin{enumerate}
	\item  $g$ restricted to $E \times \{ r_3 \}$ is the metric $g_I$.  
	\item  $g$ restricted to $E \times (r_4, \infty)$ is the metric $\c_n$.  
	\item  all sectional curvatures of $g$ lie in the interval $(-4-\e, -1 + \e)$.  
	\end{enumerate}
\end{cor}

\vskip 10pt

\begin{proof}
Just apply the procedure in the proof of Theorem \ref{thm:warping construction 1} backwards.  
Note that this metric is only defined on $E \times [r_3, \infty)$ for some $r_3 > 0$ large since $g_I$ may not have almost pinched negative curvature for small values of $r$.  
\end{proof}

\vskip 20pt

\subsection{The $d$-fold integrable complex hyperbolic metric ${\bf  g_{d}}$}
Let $d \geq 2$ be a positive integer.  
The {\it $d$-fold integrable complex hyperbolic metric} $g_{d}$ is defined as $h(r) = \cosh(r)$ and $v(r) = d \sinh(2r) = 2d \sinh(r) \cosh(r)$ in equation \eqref{eqn:mu}, and setting all structure constants defined in \eqref{eqn:structure constants 1} identically equal to zero.  
So as a metric, we have that
	\begin{equation*}
	g_{d} = \cosh^2(r) \c_{n-1} + \frac{d^2}{4} \sinh^2(2r) d \th^2 + dr^2.
	\end{equation*}
Note that this metric has a cone angle of $2 d \pi$ about the core copy of $\C \H^{n-1}$.  
For the metric $g_d$ to be complete, we need to restrict it to $E \times (r_0, \infty)$ for some $r_0 > 0$.

\begin{remark}
The curvature formulas for $g_d$ are exactly the same as the curvature formulas for $g_I$.  
When all structure constants are $0$ in Theorem \ref{thm:new curvature equations}, then the extra ``d" that is inserted in any denominator with a $v$ will cancel with the same $d$ in the numerator that comes from either a $v'$ or a $v''$.  
\end{remark}

Therefore, by Lemma \ref{thm:pinched curvature} we know that, given $\e > 0$, all sectional curvatures of $g_d$ lie in $(-4-\e, -1+\e)$ for $r$ sufficiently large.

\vskip 20pt

\begin{theorem}\label{thm:warping construction 2}
Let $\e > 0$.  
Then there exist positive constants $r_2 < r_3$ and a Riemannian metric $g$ defined on $E \times [r_2, \infty)$ such that
	\begin{enumerate}
	\item  $g$ restricted to $E \times \{ r_2 \}$ is the metric $g_I$.  
	\item  $g$ restricted to $E \times (r_3, \infty)$ is the metric $g_d$.  
	\item  all sectional curvatures of $g$ lie in the interval $(-4-\e, -1 + \e)$.  
	\end{enumerate}
\end{theorem}

\vskip 10pt

\begin{proof}
This proof is very similar to the construction in Section 2 of \cite{GT}.  
Define a function $\s : \R \to [1,d]$ which satisfies
	\begin{itemize}
	\item  $\s(r) = 1$ for all $r \leq r_2$.  
	\item  $\s(r) = d$ for all $r \geq r_3$.  
	\item  $\s'(r) < \d$ and $\s''(r) < \d$ for all $r$.  
	\end{itemize}
Given any fixed $\d > 0$, such a function $\s$ clearly exists provided $r_2$ and $r_3 - r_2$ are sufficiently large.  

We now define the metric $g$ on $E \times [r_2, \infty )$  to be the metric in Theorem \ref{thm:new curvature equations} with the substitutions $h(r) = \cosh(r)$ and $v(r) = \s(r) \sinh(2r) = 2 \s(r) \sinh(r) \cosh(r)$.  
Using the notation of Theorem \ref{thm:new curvature equations}, up to the symmetries of the curvature tensor, the nonzero components of the sectional curvature tensor with respect to $g$ are
	\begin{align*}
	&R_{i,j,i,j} = - 1   \hskip 40pt 	R_{i,2n,i,2n} = - 1  \\	
	&R_{i,2n-1,i,2n-1} = - 1 - \tanh^2(r) - \frac{\s'}{\s} \tanh(r)	\\
	&R_{i,i+1,i,i+1} = -1 - \frac{3}{\cosh^2(r)}	\hskip 40pt	R_{2n-1,2n,2n-1,2n}= - 4 - \frac{4 \s'}{\s} \coth(r) - \frac{\s''}{\s}  \\
	&R_{i, i+1, 2n-1, 2n} = R_{i, 2n-1, i+1, 2n} = R_{i, 2n, i+1, 2n-1} = 0  \\
	&R_{i, i+1, k, k+1} = 2R_{i, k, i+1, k+1} = -2R_{i, k+1, i+1, k} = - \frac{2}{\cosh^2(r)}
	\end{align*}
Since $\tanh(r)$ and $\coth(r)$ both approach $1$ as $r$ approaches $\infty$, one sees that these curvature formulas approach those of the metric $g_I$ for $\s'$ and $\s''$ sufficiently small.  
That is, for $\d$ chosen sufficiently small.  
Thus, all sectional curvatures of $g$ will lie in the interval $(-4 - \e, -1 + \e)$ for appropriately chosen $r_2$, $r_3$, and $\d$.  
\end{proof}

\begin{proof}[Proof of Theorem \ref{thm:main theorem}]
Via the discussion in the Introduction, all that is left to prove is the construction of the metric $g$ on the $d$-fold ramified cover of $\C \H^n$ about $\C \H^{n-1}$ with all curvatures almost $1/4$-pinched.
This really follows immediately from Theorem \ref{thm:warping construction 1}, Theorem \ref{thm:warping construction 2}, and Corollary \ref{cor:warping construction 3}, but we give a quick rigorous argument here.  

Define $r_1$ as in Theorem \ref{thm:warping construction 1}.  
Choose $r_2$ to be the larger of those from Theorem \ref{thm:warping construction 1} and Theorem \ref{thm:warping construction 2} and $r_3$ to be the larger of those from Theorem \ref{thm:warping construction 2} and Corollary \ref{cor:warping construction 3}.  
Finally, choose $r_4$ as in Corollary \ref{cor:warping construction 3} as needed for the chosen $r_3$.  
Let $g_1$ be the metric from Theorem \ref{thm:warping construction 1}, let $g_2$ be the metric from Theorem \ref{thm:warping construction 2}, and let $g_3$ be the metric from Corollary \ref{cor:warping construction 3}.  

Define a metric $g$ on $E \times (0,r_3]$ by
	\[ g =  \begin{cases} 
      	 \c_n & \text{if } r \leq r_1 \\
     	 g_1 & \text{if } r_1 \leq r \leq r_2 \\
      	 g_2 & \text{if } r_2 \leq r \leq r_3
   	\end{cases}
	\]
and note that $g$ is well-defined when $r = r_1, r_2$, and that all sectional curvatures of $g$ lie in the interval $(-4-\e, -1+ \e )$.  

At this point, the metric $g$ restricted to $E \times \{ r_3 \}$ has total angle of $2 d \pi$ about the core copy of $\R^{2n-2}$.  
Subdivide $\S^1$ into $d$ equidistant arcs, and label them $A_1, \hdots, A_d$.  
Then for each $1 \leq \ell \leq d$, the metric $g$ restricted to $\R^{2n-2} \times A_\ell \times \{ r_3 \}$ has total angle $2 \pi$ about $\R^{2n-2}$.  
Moreover, $g$ restricted to a sector $A_\ell$ is equal to $g_{I}$ over all of $\S^1$.
So to finish our construction of $g$ we restrict our attention to $\R^{2n-2} \times A_\ell \times [r_3, \infty)$. 

On $\R^{2n-2} \times A_\ell \times [r_3, \infty)$ define $g = g_3$ for all $1 \leq \ell \leq d$.
Note that this agrees with the previous definition for $g$ on $E \times \{ r_3 \}$, and all sectional curvatures of $g$ remain in $(-4 - \e, -1 + \e)$.  
As long as the ramification locus of $X$ has a normal injectivity radius of at least $r_3$, this metric $g$ will descend to a well-defined Riemannian metric on $X$ that is $\e$-close to being negatively $1/4$-pinched.  

\end{proof}

\end{document}